\documentclass[12pt]{article}
\usepackage{hyperref}
\usepackage{cmap}                        
\usepackage[cp1251]{inputenc}            
\usepackage[english]{babel}
\usepackage[left=1in,right=1in,top=1in,bottom=1.5in]{geometry} 
\usepackage{amssymb,amsmath, amsthm, amscd,ifthen}
\usepackage{graphicx}

\newtheorem*{thm*}{Theorem}
\newcommand{\ff}{{\mathcal F}}
\newcommand{\G}{{\mathcal G}}
\newcommand{\aaa}{{\mathcal A}}

\newtheorem*{cla*}{Claim}

\newtheorem{thm}{Theorem}

\newtheorem{cla}[thm]{Claim}
\newtheorem{ex}[thm]{Example}
\newtheorem{prb}{Problem}

\newtheorem{cor}[thm]{Corollary}
\date{}

\newtheorem{prop}[thm]{Proposition}
\newtheorem{obs}[thm]{Observation}

\newtheorem{defn}{Definition}

\DeclareMathOperator{\E}{\mathrm E}
\DeclareMathOperator{\s}{\mathcal S}

\date{}
\title{Uniform intersecting families with large covering number}

\author{Peter Frankl\footnote{R\'enyi Institute, Budapest, Hungary and Moscow Institute of Physics and Technology, Russia, Email: {\tt peter.frankl@gmail.com}} and Andrey Kupavskii\footnote{G-SCOP, CNRS, University Grenoble-Alpes, France; Moscow Institute of Physics and Technology, Russia; Saint Petersburg State University, Russia. Email: {\tt kupavskii@yandex.ru}.}}
\begin{document}
\maketitle
\begin{abstract}
 A family $\mathcal F$ has covering number $\tau$ if the size of the smallest set intersecting all sets from $\mathcal F$ is equal to $\tau$. Let $M(n,k,\tau)$ stand for the size of the largest intersecting family $\mathcal F$ of $k$-element subsets of $\{1,\ldots,n\}$ with covering number $\tau$. It is a classical result of Erd\H os and Lov\'asz that $M(n,k,k)\le k^k$ for any $n$. In this short note, we explore the behaviour of $M(n,k,\tau)$ for $n<k^2$ and large $\tau$. The results are quite surprising: For example, we show that
 $$M(n,k,\tau)\begin{cases}
                =(1-o(1)){n-1\choose k-1}, & \mbox{if }  n = \lfloor k^{3/2}\rfloor \text{ and } \tau\le k-k^{3/4+o(1)} \text{ as } k\to\infty\\
                <e^{-ck^{1/2}}{n\choose k}, & \mbox{if }   n = \lfloor k^{3/2}\rfloor \text{ and } \tau>k-\frac 12k^{1/2}.
              \end{cases}$$
\end{abstract}

\section{Introduction}
Let $[n] = \{1,2,\ldots, n\}$ be the standard $n$-element set, $2^{[n]}$ its power set and, for $0\le k\le n$, let ${[n]\choose k}$ denote the collection of all $k$-subsets of $[n]$. We also use notation $[a,b] = \{a,a+1,\ldots, b\}$. Subsets of $2^{[n]}$ are called {\it families}.
A family is called {\it intersecting} if $F\cap F'\ne \emptyset$ for all $F, F'\in \ff$. Much of extremal set theory developed around the Erd\H os--Ko--Rado theorem:
\begin{thm}[Erd\H os--Ko--Rado theorem \cite{EKR}]
  Let $n\ge 2k>0$. Suppose that $\ff\subset {[n]\choose k}$ is intersecting. Then
  \begin{equation}\label{eq1.1}
    |\ff|\le {n-1\choose k-1}.
  \end{equation}
\end{thm}

The theorem is tight, as is witnessed by the following simple example of an intersecting family, called the {\it full star}: $\s_x:= \{S\subset {[n]\choose k}: x\in S\}$, where $x\in [n]$. Any subfamily of a full star is called a {\it star}, and the element $x$ is called the {\it center} of the star.
By now there are myriads of results having their origin in the Erd\H os--Ko--Rado theorem. One of the important directions deals with stability of \eqref{eq1.1}. The idea is to impose further conditions on $\ff$ in order to quantify the difference between $\ff$ and any star, and to improve the bound \eqref{eq1.1} accordingly. One of the most natural such measures of difference is the covering number of $\ff$:
\begin{defn}
For a family $\ff$ consisting of non-empty subsets and a set $T$, we say that $T$ is a {\em cover} or {\em transversal} for $\ff$ if $F\cap T\ne\emptyset $ holds for all $F\in \ff$. The {\em covering number} $\tau(\ff)$ is the minimum of $|T|$ over all covers of $\ff$.
\end{defn}
Note that if $\ff$ is intersecting then each $F\in \ff$ is a cover.

The goal of this paper is to investigate, to which extent we can hope to improve \eqref{eq1.1} if we impose a lower bound on the covering number of $\ff$. For reasons that will become clear later we are most interested in the regime $2k<n<k^2$.

\begin{obs}
  Suppose that $\ff\subset {[n]\choose k}$ is intersecting. Then
  \begin{equation}\label{eq1.2}
    1\le \tau(\ff)\le k.
  \end{equation}
\end{obs}
Note that $\tau(\ff) = 1$ if and only if all members of $\ff$ contain a fixed element $x$, i.e., if $\ff$ is a star. Let us recall the following fundamental result.
\begin{thm}[Hilton--Milner theorem \cite{HM}] Let $n>2k\ge 4$. Suppose that $\ff\subset {[n]\choose k}$ is intersecting, and $\tau(\ff)\ge 2$. Then
\begin{equation}\label{eq1.30} |\ff|\le {n-1\choose k-1}-{n-k-1\choose k-1}+1.
\end{equation}
\end{thm}
This inequality is tight. To show the families that attain equality in \eqref{eq1.30} let us first define a more general class of families. For two integers $p,q$, $p\ge q$, we use notation $[p,q] =\{p,p+1,\ldots, q\}$.

\begin{ex} Let $n,k,t$ be integers, $n\ge 2k$, $k> t\ge 2$. Define $$\aaa_t := \aaa_t(n,k):={[2,k+t]\choose k}\cup \Big\{A\in{[n]\choose k}: 1\in A, |A\cap [2,k+t]|\ge t\Big\}.$$
\end{ex}
Note that ${[k+t]\choose k}\subset \aaa_t$, and thus $\tau(\aaa_t)\ge t+1$. Since $[t+1]$ is a cover for $\aaa_t$, $\tau(\aaa_t) = t+1$ follows.
It is easy to check that $\aaa_t$ is intersecting and that $\aaa_1$ attains equality in \eqref{eq1.30}. As a matter of fact, for $k>3$ any extremal family for the Hilton--Milner theorem can be obtained from $\aaa_t$ by permuting the ground set $[n]$.

We should mention that during the years many different, shorter proofs were found for \eqref{eq1.30}. Let us mention a few: Frankl--F\"uredi \cite{FF}, M\"ors \cite{M}, Borg \cite{B}, Kupavskii--Zakharov \cite{KZ}, Frankl \cite{F47}.
To find the largest intersecting family with fixed $\tau\ge 3$ is a much more complicated problem. The second author determined the largest intersecting family with $\tau(\ff) \ge 3$, for $n>Ck$ \cite{Kup23} with some large absolute $C$. This improved upon an earlier result of the first author \cite{F8}, which resolved the same problem  for $n>n_0(k)$. The extremal example in this case is different from $\aaa_2$.

Let us look at the other end of the spectrum of the values of $\tau$. One early gem in extremal set theory was the seminal paper of Erd\H os and Lov\'asz \cite{EL}. Let us recall one fundamental result from this paper.

\begin{thm}
  Suppose that $\ff\subset {[n]\choose k}$ is intersecting with $\tau(\ff) = k$. Then
  \begin{equation}\label{eq1.4}
    |\ff|\le k^k.
  \end{equation}
\end{thm}
Note that the bound \eqref{eq1.4} is independent of $n$.
\begin{ex}[\cite{EL}] \label{ex1.9} Let $k\ge 2$ and  $X_1,\ldots, X_k$ be pairwise disjoint sets with $|X_i| = i,$ $1\le i\le k$. Define the family $\mathcal E_k$, $1\le i\le k$ by
$$\mathcal E_i = \big\{X_i\cup \{y_{i+1},\ldots, y_k\}: y_j\in X_j \text{ for }i<j\le k\big\}.$$
Clearly, $|\mathcal E_i| = \frac {k!}{i!}$. Set $\mathcal E = \mathcal E_1\cup\ldots \cup \mathcal E_k$. Then $|\mathcal E|=\lfloor e k!\rfloor$, $\mathcal E$ is intersecting and $\tau(\mathcal E) = k$.
\end{ex}
Let us define $M(k)$ as the maximum of $|\ff|$ over all $k$-uniform intersecting families with $\tau(\ff) = k$. by the above,
\begin{equation}\label{eq1.5}
  \lfloor e k!\rfloor \le M(k)\le k^k.
\end{equation}
For $k=2$ and $3$ it is known that the example above is optimal. This made Lov\' asz \cite{L} conjecture that it is optimal for all $k$. However, Frankl, Ota and Tokushige \cite{FOT1} disproved it for $k=4$ and then in \cite{FOT2} provided an exponential improvement on the lower bound by showing
\begin{equation}\label{eq1.6}
  \Big(\frac{k-o(1)}2\Big)^k\le M(k).
\end{equation}
Note that Example~\ref{ex1.9} has ${k+1\choose 2}$ vertices. The examples of \cite{FOT2} have roughly the same number of vertices.
Let us describe a similar, but simpler, example for even $k$.
\begin{ex} \label{ex1.99} Let $k\ge 2$ be even and  $X_1,\ldots, X_{k-1}$ be pairwise disjoint sets with $|X_i| = k/2+1,$ $1\le i\le k-1$. Let $v$ be an element of the ground set, $v\notin \cup X_i$. Define the family $\mathcal O = \mathcal O_1\cup \mathcal O_2$ as follows.
$$\mathcal O_1 = \big\{Y: v\in Y, |X_i\cap Y|=1 \text{ for each } i\in [k-1]\big\}.$$
$$\mathcal O_2 = \big\{Z: X_i\subset Z \text{ for some $i\in [k-1]$ and } |Z\cap X_j|=1 \text{ for all }j=i+1,\ldots, i+k/2-1\big\},$$
where the indices in the last expression are taken modulo $k-1$. Note that the sets in both $\mathcal O_1$ and $\mathcal O_2$ have cardinality $k$. We have $|\mathcal O| = (1+o(1))|\mathcal O_1| = (k/2+1)^{k-1}$. It is also possible to check that  $\mathcal O$ is intersecting and that $\tau(\mathcal O) = k$.
\end{ex}

A similar, albeit more general, construction, will appear in the proof of Theorem~\ref{thmlb}, and so we omit the verification of the last two properties of $\mathcal O$ here.

We should also say that the upper bound in \eqref{eq1.5} has been gradually improved, and the current record $M(k)\le k^{k-k^{1/2+o(1)}}$ is due to Zakharov \cite{Z}, improving on \cite{Fra28}.

\subsection{Results}\label{sec11}
Note that for $n =ck^2$, ${n\choose k} = C\frac {n^k}{k!}\sim (ce)^{k+o(k)} k^k$, where $C$ is some constant depending on $c$. That is, the upper bound in \eqref{eq1.5} is meaningful for $n>e^{-1} k^2$. On the other hand, for $n\ge k^2$ we conclude from \eqref{eq1.5} and the asymptotic formula above that intersecting families with covering number $k$ are exponentially small w.r.t. ${n\choose k}$. Thus, the regime $n\ge k^2$ is well-covered by the known upper and lower bounds for the case when no restriction on the size of the ground set is imposed: we know that the upper and lower bounds are of the form $(ck)^k$, where $0.5-o(1)\le c \le 1$, and both upper and lower bounds are exponentially small w.r.t. ${n\choose k}$. In this paper, we mostly focus on the situation for $n<k^2$.

\begin{thm}\label{thmub} Fix integers $n,k,\tau$ and put $\beta := k+1-\big\lfloor\frac{n-1}{k-1}\big\rfloor$. If an intersecting family $\ff$ satisfies $\tau(\ff) \ge\tau >\beta$ then we have the following bound: $$|\ff|\le {n\choose k} \prod_{i=\beta}^{\tau-1} \frac{(k-i)k}{n-i}.$$
\end{thm}
The next corollary gives numerical bounds that follow from Theorem~\ref{thmub}. In the results employing $o$-notation, we mean that we take a limit $k\to \infty$. We also note that, although we do not actually need $k$ to be particularly large, the results and the calculations are much cleaner in their asymptotic form.

\begin{cor}\label{corub}
  If $\ff\subset {[n]\choose k}$ for $n<k^2$ and $n/k\to \infty$ satisfies $\tau (\ff) = k-c\frac nk+1$ with some constant $c<1$, then $|\ff|\le e^{-(c'+o(1))n/k}{n\choose k}$, where $c' = 1-c+c\log c$.
\end{cor}
Here and in what follows $\log$ stands for natural logarithm. We use a standard convention that $c\log c = 0$ for $c=0$. It is not hard to check that $1-c+c\log c$ is positive for any $c\in (0,1)$.

Let us now turn to constructions that provide lower bounds. They are in spirit similar to the aforementioned construction of Frankl, Ota and Tokushige \cite{FOT2}.

For positive integers $k,\ell$, we define
$$f(k,\ell) = \log k+\sqrt{\log^2 k+2\ell \log k}.$$

\begin{thm}\label{thmlb} Take positive integers $n,\ell,m, k$ such that $k=1+(\ell+f_1(k,\ell)+1)(2m+1)$ for some $f_1(k,\ell)\ge f(k,\ell)$, $n\ge 2k$ and $n=1+(\ell+1)(m+2)(2m+1)$. There is an intersecting family $\ff\subset {[n]\choose k}$ with $\tau(\ff) \ge \frac{\ell+1}{\ell+f_1(k,\ell)+1}k$ and $|\ff| = (1-o(1)){n-1\choose k-1}$ as $k\to \infty$.
\end{thm}

In the numerical corollary that follows, we in some places ignore that certain values must be integers or certain equations must have integral solutions. This can be achieved by tweaking the parameters a little bit and does not affect the asymptotic results that we present.
\begin{cor}\label{corlb}
  Assume that $n,k$ are sufficiently large positive integers. We have an intersecting family $\ff\subset {[n]\choose k}$ of size $(1-o(1)){n\choose k}$ with $\tau(\ff) = \tau$ with the following values of parameters:
  \begin{itemize}
    \item[(i)] for any $c>28$ with $n \sim\frac{k^2}{c\log k}$  and $\tau \ge c'k$ for some $c'>\frac 14$. Moreover, $c'\to 1$ as $c\to \infty$.
    \item[(ii)] for any $0<\epsilon<1$ with $n \sim k^{2-\epsilon}$  and $\tau \ge k-k^{1-\epsilon/2+o(1)}$.
  \end{itemize}
\end{cor}

Denote by $M(n,k,\tau)$ the largest size of an intersecting family $\ff\subset {[n]\choose k}$ with $\tau(\ff)= \tau$.
\begin{prop}\label{prop1}
  Assume that $n\ge 2k$ and $\tau\ge 2$. Then $M(n,k,\tau-1)\ge M(n,k,\tau)$. In particular, the statements of Theorem~\ref{thmlb} and Corollary~\ref{corlb} also hold for each value of $\tau(\ff)$ that belongs to $[1,\tau']$, where $\tau'$ is the lower bound imposed on $\tau(\ff)$ in the corresponding statement.
\end{prop}

Thus, we see a sharp cutoff: for, say, $n=k^{3/2}$ there are intersecting families $\ff\subset {[n]\choose k}$ with $\tau(\ff) = k-k^{3/4+o(1)}$ and that have size essentially ${n-1\choose k-1}$, while if we ask for families with $\tau(\ff) = k-\frac 12k^{1/2}$, then the largest such family has size at most $e^{-ck^{1/2}}{n\choose k}$. We think that this is an interesting question to determine the point at which this transition happens. As an example question, let us give the following. (We again omit integer parts for simplicity.)
\begin{prb}
 What is the smallest value of $\gamma$ such that  for $n = k^{3/2}$ we have $M(n,k,k-k^{\gamma})= o\big({n-1\choose k-1}\big)$ as $k\to \infty$?
\end{prb}
Our results imply that $\frac 12\le \gamma\le \frac 34$.
\subsection{Shifted families}
Many of the proofs of theorems on intersecting families rely on {\it shifting}, an extremely useful operation whose definition can be traced back to Erd\H os, Ko and Rado \cite{EKR} (see also \cite{Fra3}).

For a family $\ff\subset {[n]\choose k}$ and $1\le i<j\le n$ define the $(i\leftarrow j)$-shift $S_{i\leftarrow j}$ by
\begin{align*}
  S_{i\leftarrow j}(\ff) &= \big\{S_{i\leftarrow j}(A): A\in \ff \big\} \cup \big\{A: S_{i\leftarrow j}(A) \text{ and } A\in \ff \big\}, \ \ \ \ \text{where} \\
  S_{i\leftarrow j}(A) &=\begin{cases}
                           \tilde A:=(A\setminus \{j\})\cup\{i\}, & \mbox{if } A\cap \{i,j\} =\{j\}\\
                           A, & \mbox{otherwise}.
                         \end{cases}
\end{align*}
We say that $\ff$ is {\it shifted} if $S_{i\leftarrow j}(\ff) = \ff$ for all $1\le i<j\le n$. Alternatively, $\ff$ is shifted  if for all $F\in \ff$ and $1\le i<j\le n$ satisfying $i\notin F$, $j\in F$ the set $(F\setminus \{j\})\cup\{i\}$ is a member of $\ff$.

We show that $\aaa_{t-1}$ is maximal among shifted intersecting families with covering number at least $t$.

\begin{thm}\label{thmsh}
Suppose that $n\ge 2k$, $k\ge t>1$, $\ff\subset {[n]\choose k}$ is intersecting, $\tau(\ff)\ge t$. If $\ff$ is shifted then
\begin{equation}\label{eq1.3}
  |\ff|\le |\aaa_{t-1}|.
\end{equation}
\end{thm}

It is interesting to compare the general and shifted cases. It is not difficult to check that $|\aaa_{t}| = (1-o(1)){n-1\choose k-1}$ for $t\le (1-\epsilon)\frac {k^2}n$, provided $\epsilon>0$ and $\frac {k^2}n\to \infty$. Indeed, among all sets containing $1$, any set that intersects $[2,k+t]$ in at least $t$ elements belongs to $\aaa_t$. If we take a random $(k-1)$-element subset of $[2,n]$, then the expected intersection of it with $[2,k+t]$ is $\frac{(k+t-1)(k-1)}{n-1}$, and a standard concentration inequality for hypergeometric distributions (cf. \eqref{eqconcent}) implies that almost all $(k-1)$-subsets of $[2,n]$ have intersection at least $(1-\epsilon)\frac{k^2}n$ with $[2,k+t]$.

On the other hand, concentration inequalities applied in the other direction imply that, for $n/k\to \infty$ and $t\ge (1+\epsilon)\frac{k^2}n$ the size of $\aaa_t$ is at most $e^{-c(\epsilon) k^2/n}{n\choose k}$, where $c(\epsilon)>0$ is some constant depending on $\epsilon$.

Taking a concrete example of $n=k^{3/2}$, we see that the  transition from large to exponentially small family sizes in the general and shifted cases happen at very different values of $\tau$. In the former case, we have an intersecting family with $|\ff| = (1-o(1)){n-1\choose k-1}$ for $\tau(\ff) = k-k^{3/4}$, while in the latter case the same can be guaranteed for $\tau(\ff) \sim k^{1/2}$ only.

We prove Theorem~\ref{thmsh} in the remainder of this subsection.
First, note that if $\ff$ satisfies the conditions of the theorem then $U:=\{t,t+1,\ldots, t+k-1\}$ is contained in $\ff$. Indeed, there must be a set $A\in \ff$ such that the smallest element in $A$ is at least $t$, otherwise $[t-1]$ is a cover. On the other hand, $\ff$ being shifted implies that $U\in \ff$. This, in turn, implies that ${[t+k-1]\choose k}\subset \ff$.

Let us recast the following theorem of the first author.
\begin{thm}[\cite{F65}]
  Let $n\ge 2k>m\ge k\ge 1$.  Let $\ff\subset {[n]\choose k}$ be an intersecting family such that ${[m]\choose k}\subset \ff$. Then \begin{equation*}\label{eq2.1}
    |\ff|\le {m\choose k}+\sum_{i=m-k+1}^{k-1}{m-1\choose i-1}{n-m\choose k-i}.
  \end{equation*}
\end{thm}
It is not difficult to see that the right hand side is actually $|\aaa_{m-k}|$.

Applying this theorem with $m = k+t-1$ to $\ff$, we conclude that $|\ff|\le |\aaa_{t-1}|$.

\section{Proofs}
\begin{proof}[Proof of Theorem~\ref{thmub}]
  For any $X$, let us denote $\ff(X):= \{F\in \ff: X\subset F\}$. Put
  $$\alpha_i:= \max\Big\{\frac{|\ff(X)|}{{n-i\choose k-i}}\ :\  X\in {[n]\choose i}\Big\}.$$
  In words, $\alpha_i$ is the largest density of $\ff(X)$ over all sets $X$ of size $i$. In particular, $\alpha_0$ is the density of $\ff$. Note that $\alpha_i\le 1$ for any $i$ and also that $\alpha_i\le \alpha_{i+1}$ via simple averaging. Indeed, let $X$ be such that $|X|= i$ and $\alpha_i = \frac{|\ff(X)|}{{n-i\choose k-i}}$. Next, we count the sets from $\ff$ containing $X$ in two ways. On the one hand, it is  $|\ff(X)|$, and, on the other hand, we fix an extra element $y\in [n]\setminus X$ and count the sizes of $\ff(X\cup \{y\})$. There are $n-|X|$ such choices of $y$, and each set from $\ff(X)$ is included into $k-|X|$ such families. We get $|\mathcal F(X)| = \frac{1}{k-i}\sum_{y\in [n]\setminus X} |\mathcal F(X\cup\{y\})|\le \frac{n-i}{k-i}|\ff(X\cup \{y_{max}\})|, $ where $y_{max}$ is such that $|\ff(X\cup \{y_{max}\}|$ is the largest. Renormalizing, we get
  $$\alpha_i = \frac{|\ff(X)|}{{n-i\choose k-i}}\le \frac{|\ff(X\cup \{i_{max}\})|}{{n-i-1\choose k-i-1}}\le \alpha_{i+1}.$$

   We claim that for any $|X|<\tau$ there exists a set $A\in \ff$ disjoint with $X$ such that
  \begin{equation}\label{eq1}\ff(X)\subset \bigcup_{j\in A} \ff(X\cup\{j\}).\end{equation}
  Indeed, since $X$ is not a cover of $\ff$, there should be a set $A\in \ff$ that is disjoint with $X$. But since every set $F\in \ff$ intersects $A$, we have the inclusion above.

Thus, \eqref{eq1} implies ${n-i\choose k-i}\alpha_i\le k{n-i-1\choose k-i-1}\alpha_{i+1}$ for each $i<\tau$, which is equivalent to $$\alpha_i\le \frac{k(k-i)}{n-i}\alpha_{i+1}.$$
This inequality is a non-trivial improvement of $\alpha_i\le \alpha_{i+1}$ whenever the fraction in front of $\alpha_{i+1}$ is at most $1$, which happens for $i\ge \frac{k^2-n}{k-1} = k+1-\frac{n-1}{k-1}$. Note that the upper integer part of the last number is exactly $ \beta$ from the formulation of the theorem. Applying the displayed bound for $i= \beta,\ldots, \tau-1$, we get that
$$\alpha_\beta\le \prod_{i=\beta}^{\tau-1} \frac{(k-i)k}{n-i} \alpha_\tau\le \prod_{i=\beta}^{\tau-1} \frac{(k-i)k}{n-i}.$$

Now, note that $\alpha_0\le \alpha_1\le\ldots \le \alpha_{\beta}$. Thus, we have $|\ff|/{n\choose k} = \alpha_0\le \alpha_\beta$, and the desired bound is obtained by combining with the displayed inequality.
\end{proof}

\begin{proof}[Proof of Corollary~\ref{corub}]
  We only have to bound from above the product $\prod_{i=\beta}^{k-c\frac nk} \frac{(k-i)k}{n-i}$, where $\beta := k+1-\frac{n-1}{k-1}$. Recall that each term of this product is smaller than $1$. In the first inequality below we use this fact and that $\frac{n-1}{k-1}\ge \frac nk$. In the third inequality we use that $\big(\frac ne\big)^n\le n!\le n\cdot \big(\frac ne\big)^n$, valid for all $n\ge 5$. We get
  $$\prod_{i=\beta}^{k-c\frac nk} \frac{(k-i)k}{n-i}\le \prod_{i=k+1-\frac nk}^{k-c\frac nk} \frac{(k-i)k}{n-i}=\frac{k^{(1-c)\frac nk}(\frac nk-1)!}{\big(c\frac nk-1\big)!\prod_{i=k+1-\frac nk}^{k-c\frac nk}(n-i)} $$
  $$\le C\frac{k^{(1-c)\frac nk} \big(\frac{n}{k}\big)!}{\big(c\frac{n}{k}\big)!(n-k)^{(1-c)\frac nk}}\le C\Biggl(\frac{k }{n-k}\Biggr)^{(1-c)\frac nk}\cdot \frac nk\frac{\big(\frac n{ek}\big)^{n/k}}{\big(\frac {cn}{ek}\big)^{cn/k}}$$
  $$= C\frac nk \Biggl(\frac{n}{n-k}\Biggr)^{(1-c)\frac nk}\big(c^{-c}e^{(c-1)}\big)^{\frac nk} = e^{(c-1-c\log c+o(1))\frac nk},$$






  where $C>1$ is some constant. We only used that  $cn\ge 5k$ in the third inequality. 
\end{proof}

\begin{proof}[Proof of Theorem~\ref{thmlb}] Put $s= (\ell+1)(m+2)$ and note $k>s>2(\ell+1)$.

Fix $2m+1$ pairwise disjoint sets $W_1,\ldots, W_{2m+1}\in {[2,n]\choose s}$. Let $\G_1\subset{[n]\choose k}$ be the family of all $k$-element subsets of $[n]$ that contain $1$ and intersect each $W_i$ in at least $\ell+1$ elements, i.e.
$$\mathcal G_1 = \Big\{\{1\}\sqcup B: B\in {[2,n]\choose k-1},\ |B\cap W_i|\ge \ell+1, \text{ for each }i\in [2m+1]\Big\}.$$
Let $\G_2$ be the family of all $k$-element subsets of $n$ that for some $i$ contain at least $s-\ell$ elements in $W_i$ and at least $\ell+1$ element in each $W_j$, where $j = i+1,\ldots, i+m$ and indices are taken modulo $2m+1$, i.e.,
{\small $$\mathcal G_2 = \Big\{A\in {[2,n]\choose k}: |A\cap W_i|\ge s-\ell \text{ for some }i\in [2m+1]\text{ and } |A\cap W_{i+j}|\ge \ell+1, \text{ for each }j\in [m]\Big\}.$$}
We note that, first, $n-1 = (2m+1)|W_i|$, and thus the sets $W_i$ will fit into $[2,n]$. Second, note that $k\ge1+(\ell+1)(2m+1) =(s-\ell)+(\ell+1)m$, and thus the families $\G_1,\G_2$ are defined correctly and uniquely.

Put $\G:=\G_1\cup \G_2$. Let us show that $\G$ is intersecting. Indeed, any two sets in $\G_1$ intersect because they all contain $1$. A set  $A\in\G_1$ intersects a set $B\in \G_2$ because there is an $i$ such that $|B\cap W_i|\ge s-\ell = |W_i|-\ell$, and $|A\cap W_j|\ge \ell+1$ for any $j$, in particular $|A\cap W_i|\ge \ell+1$. If $A\cap B\cap W_i$ is empty then $|W_i|\ge |A\cap W_i|+|B\cap W_i|\ge s+1,$ which contradicts $|W_i| = s$. Finally, consider two sets $B,C\in \G_2$. Assume that $|B\cap W_i|\ge s-\ell$ and $|C\cap W_j|\ge s-\ell$. Then either $i-j$ or $j-i$ belongs to  $\{0,\ldots,m\}$ modulo $2m+1$. Assume that the former holds. But then $|C\cap W_i|\ge \ell+1$, and arguing similarly, $B\cap C\cap W_i\ne \emptyset.$

Next, we aim to determine $\tau(\G)$. Let $C$ be a subset of $[n]$ with $|C|\le k$. We claim the following.
\begin{cla}\label{cla15}
  $C$ intersects all the members of $\mathcal G_1$ if and only if $C$ contains $1$ or $|C\cap W_i|\ge s-\ell$ for some $i\in [2m+1]$.
\end{cla}
\begin{proof}
  Suppose that $C$ does not intersect all members of $\G_1$. Then there exists at least one $A\in \G_1$ such that $A\cap C = \emptyset$. This means that $1\notin C$. At the same time, for each $i\in [2m+1]$ we have $|A\cap W_i|\ge \ell+1$. We get $s = |W_i|\ge |A\cap W_i|+|C\cap W_i|$, and so $|C\cap W_i|\le s-\ell-1$ for each $i\in [2m+1]$. Conversely, suppose $C$ does not contain $1$ and $|C\cap W_i|\le s-\ell-1$ for each $i\in [2m+1]$. Choose for each $i\in [2m+1]$, $B_i\subset W_i\setminus C$ with $|B_i|\ge \ell+1$ and construct a set $B\subset [2,n]$, where $\sqcup_{i=1}^{2m+1}B_i\subset B$, $|B|=k-1$ and $B\setminus \sqcup_{i=1}^{2m+1}W_i$ is disjoint from $C\setminus\sqcup_{i=1}^{2m+1}W_i$. Here $\{1\}\sqcup B\in \G_1$ and it is disjoint from $C$. This establishes the claim.
\end{proof}

\begin{cla}\label{cla16}
  $C$ intersects all the members of $\G_2$ if and only if for each $i\in[2m+1]$ such that $|C\cap W_i|\le \ell$ there exists $j\in [m]$ such that 
   $|C\cap W_{i+j}|\ge s-\ell.$
\end{cla}
\begin{proof}
  Suppose that $C$ does not intersect some $A\in \G_2$. Since $A\in\G_2$ then there is an $i$ such that $|A\cap W_i|\ge s-\ell$ and $|A\cap W_{i+j}|\ge \ell+1$ for each $j\in[m]$. Arguing as before we get that $|C\cap W_i|\le \ell$ and $|C\cap W_{i+j}|\le s-\ell -1$ for each $j\in [m]$. Conversely, suppose that there is an $i$ such that $|C\cap W_i|\le \ell$ and $|C\cap W_{i+j}|\le s-\ell-1$ for each $j\in[m]$. Choose $B_i\subset W_i\setminus C$ with $|B_i|\ge s-\ell$ and $B_{i+j}\subset W_{i+j}\setminus C$ with $|B_{i+j}|\ge \ell+1$. Construct a set $B\subset [2,n]$, where $\sqcup_{j=0}^{m}B_{i+j}\subset B$, $|B|=k-1$ and $B\setminus \sqcup_{i=0}^{m}W_i$ is disjoint from $C\setminus\sqcup_{i=0}^{m}W_i$. Here $B\in \G_1$ and it is disjoint from $C$. This establishes the claim.
\end{proof}


\begin{cla}\label{cla17} If $C$ is a cover of minimal size 
then it should satisfy one of the following:
\begin{itemize}
  \item[(i)] $C=\{1\}\sqcup \Big(\bigsqcup_{i=1}^{2m+1}C_i\Big),$ where $C_i\in {W_i\choose \ell+1}$. 
  \item[(ii)] $C = \bigsqcup_{i=0}^m C_{i+j}$ where $C_i\in {W_i\choose s-\ell}$ and $C_{i+j}\in {W_{i+j}\choose \ell+1}$ for each $j\in [m]$ 
  \item[(iii)] $C = C_i\sqcup C_{i+m+1}$, where $C_j\in {W_j\choose s-\ell}$ for $j = i,i+m+1$.
\end{itemize}
\end{cla}
\begin{proof}
The first thing to note is that each of these three types of sets are actually covers by Claims~\ref{cla15},~\ref{cla16}. In order to check the sufficient condition from Claim~\ref{cla16} for (ii), note that for any $j'$, $m<j'\le 2m$ we put $j = 2m+1-j'$, $j\le m$, and get that $C\cap W_{i+j'+j} = C\cap W_i$, $|C\cap W_i|\ge s-\ell$. In order to check the sufficient condition from Claim~\ref{cla16} for (iii), note that for any $i'\in [2m+1]$ we can find $j\in [0,m]$ such that $i'+j\in \{i, i+m+1\}$, where indices are as usual taken modulo $2m+1$.

  Assume that for all $i\in [2m+1]$ we have $|C\cap W_i|\le s-\ell$. Then by Claim~\ref{cla15} $1\in C$ and by Claim~\ref{cla16} some set $C_i\subset C\cap W_i$ must satisfy $|C_i|=\ell+1$ for all $i$. 
  By minimality, 
  $C=\{1\}\sqcup \bigsqcup_{i=1}^{2m+1}C_i$. Thus, $C$ is as in (i).

Next, assume that for some $i$ and a set $C_i\subset C\cap W_i$ we have $|C_i|=s-\ell$. If for each $j\in [m]$ we have $|C\cap W_{i+j}|\ge \ell +1$ then by minimality we get an example (ii). Otherwise, there is a $j\in [m]$ such that $|C\cap W_{i+j}|\le \ell$. Then by Claim~\ref{cla16} there is $j'\le 2m$ such that $|C\cap W_{i+j'}|\ge s-\ell$. By minimality, $C\cap W_{i'} = \emptyset$ for all $i'\in [2m+1]\setminus \{i, i+j'\}$. If $j'\ne m+1$ then the condition of Claim~\ref{cla16} fails for either $i' =i+1$ or $i'= i+j'+1$. Thus, $j' = m+1$ and $C$ must be as in (iii).
\end{proof}


We get that $$\tau(\G) = \min \{1+(\ell+1)(2m+1),(s-\ell)+(\ell+1)m,2(s-\ell)\} = 1+(\ell+1)(2m+1).$$ Note that $\tau(\G)-1 =\frac{\ell+1}{\ell+f_1(k,\ell)+1}(k-1)$ and thus
$$\tau(\G)\ge \frac{\ell+1}{\ell+f_1(k,\ell)+1}k -\frac{\ell+1}{\ell+f_1(k,\ell)+1}+1> \frac{\ell+1}{\ell+f_1(k,\ell)+1}k$$

The final step is to lower bound the size of $\G$. Let $\eta$ be a random variable, equal to the intersection of a uniformly random $(k-1)$-element subset of $[2,n]$ with a given set that is a subset of $[2,n]$ of size $s$.
Note that $$\E \eta =\frac{s(k-1)}{n-1} = \ell+f_1(k,\ell)+1.$$ For the rest of the paragraph, assume that $f_1(k,\ell)$ is chosen so that \begin{equation}\label{eqass}\Pr[\eta\ge \ell+1]\ge 1-o\Big(\frac 1{2m+1}\Big)l.\end{equation} For each $i\in[2m+1]$ let $\eta_i$ be the intersection size of a uniformly random $(k-1)$-element subset of $[2,n]$ with $W_i$. Note that $\eta_i$ is distributed as $\eta$. Using union bound, we can conclude that with probability $1-(2m+1)o(\frac 1{2m+1}) = 1-o(1)$ we have $\eta_i\ge \ell+1$ for each $i\in [2m+1]$. 
In counting terms, all but $o\big({n-1\choose k-1}\big)$ sets $G\in {n-1\choose k-1}$ intersect each of $W_1,\ldots, W_{2m+1}$ in at least $\ell+1$ elements. For each such set $G$, the set $\{1\}\cup G$ is included in $\G_1$, and so  $|\G_1| = (1-o(1)){n-1\choose k-1}$. Thus the conclusion of the theorem is valid for the family $\G$, provided that \eqref{eqass} holds true.

We are only left to verify that the choice of $f_1(k,\ell)$ allows to guarantee the validity of the assumption \eqref{eqass}. Note that $\eta$ is distributed according to a hypergeometric distribution, and the following inequality holds (cf. \cite[Theorem~2.10 and~2.1]{JLR}):
\begin{equation}\label{eqconcent}\Pr[\eta\le \E\eta-t]\le e^{-\frac{t^2}{2\E\eta}}.\end{equation}
In our case, $\E\eta = \ell+f_1(k,\ell)+1$ and we substitute $t = f_1(k,\ell)+1$. Note that $2m+1  =o(k)$,\footnote{this is clear from the form of $f(k,\ell)$} and thus, it is sufficient to show that
$$e^{-\frac{(f_1(k,\ell)+1)^2}{2(\ell+f_1(k,\ell)+1)}} \le \frac 1k.$$
This is implied by $\frac{(f_1(k,\ell))^2}{2(\ell+f_1(k,\ell))}\ge \log k$. Rewriting and solving the resulting quadratic inequality, we get that the inequality holds for $f_1(k,\ell)\ge \log k+\sqrt{\log^2 k+2\ell \log k} = f(k,\ell)$.
\end{proof}

\begin{proof}[Proof of Corollary~\ref{corlb}]
  (i) If we substitute $\ell=c_1\log k$ with $c_1\ge 1$ and $f_1(k,\ell) = f(k,\ell)$ in Theorem~\ref{thmlb}, we get that $f(k,\ell)=\big(1+\sqrt{1+2c_1}\big)\log k$ and thus $\tau(\ff)>\frac{c_1}{c_1+1+\sqrt{1+2c_1}}k$. At the same time, $\frac n{k^2} = (1+o(1))\frac{\ell}{2(\ell+f(k,\ell))^2} = (1+o(1))\frac{c_1}{2(c_1+1+\sqrt{1+2c_1})^2\log k}$. In particular, for $c_1 = 1$ this gives
$\frac n{k^2}> \frac 1{28 \log k}$ and $\tau(\ff)>\frac 1 4 k.$

(ii) If we apply Theorem~\ref{thmlb} with $\ell\sim \frac 12 k^{\epsilon}$ and  $f_1(k,\ell) = f(k,\ell)$ then we get that $f(k,\ell) \sim \sqrt{2\ell \log k}$. Next, $\frac{n}{k^2} =(1+o(1))\frac{1}{2\ell} = (1+o(1))k^{-\epsilon}.$ Finally,
$$\tau(\ff)\ge \frac{\ell}{\ell+(1+o(1))\sqrt{2\ell\log k}} k = (1-\ell^{-1/2+o(1)})k = \big(1-k^{-\epsilon/2+o(1)}\big)k. \qedhere $$
\end{proof}

\begin{proof}[Sketch of the proof of Proposition~\ref{prop1}]

  In order to prove this statement, it is sufficient  to show that we can transform each intersecting family $\ff$ with $\tau(\ff) =\tau$  into a subfamily of $\aaa:=\{A\in{[n]\choose k}: 1\in A\}$, such that
  that: (i) the size of the family at each step remains the same; (ii) the family stays intersecting at each step; (iii) the value of $\tau(\ff)$ can decrease by at most $1$ at each step. That is, the argument is using `discrete continuity': along this series of transformations, the intermediate families take all possible values of the covering number between $1$ and $\tau(\ff)$.

  First, we note that these three properties are satisfied by shifting. Indeed, the only non-standard property of shifting is (iii), and this easily follows from the fact that the family $\ff\setminus S_{j\to i}(\ff)$ is contained in the family $\{F: j\in F\}$, and thus has 
  has covering number $1$. Thus, we can start with a family $\ff$ guaranteed by Theorem~\ref{thmlb} and apply $(j\to i)$-shifts to it until we either reach the desired value of $\tau(\ff)$ or the family is shifted.

  Next, if we have not yet reached the desired value of the covering number, we can use the transformations that were used in the paper by Kupavskii and Zakharov \cite{KZ} and that allowed to gradually transform each intersecting family into a subfamily of $\aaa$ while maintaining (i)-(iii). Note that (iii) was not explicitly verified in \cite{KZ}, but it follows immediately from the definition of the transformation: at each step we replace a certain subfamily of $\ff$ that has covering number $1$ by another subfamily (as it is the case for shifting, which we have seen above). Thus, the covering number cannot decrease by more than $1$ at any step (as in the case of shifts, it actually cannot change by more than 1 at each step).

The second part of the statement follows immediately from the first part.
\end{proof}

{\sc Acknowledgements: } We thank the referees for carefully reading the text and suggesting numerous changes that helped to improve the presentation of the paper. This work was supported by a grant for research centers in the field of artificial intelligence, provided by the Analytical Center for the Government of the Russian Federation in accordance with the subsidy agreement (agreement identifier 000000D730321P5Q0002) and the agreement with the Moscow Institute of Physics and Technology dated November 1, 2021 No. 70-2021-00138.

\end{document}